\documentclass[a4paper,11pt]{article}

\usepackage[utf8]{inputenc}
\usepackage[T1]{fontenc}
\usepackage[english]{babel}
\usepackage{fullpage}
\usepackage{xspace}
\usepackage{amsmath}
\usepackage{amsthm}
\usepackage{amssymb}
\usepackage[bookmarks=true,
            bookmarksnumbered=true,
            pagebackref=true,
            colorlinks=true,
            linkcolor=blue,
            citecolor=red]{hyperref}

\title{A weak variant of Hindman's Theorem stronger than Hilbert's Theorem\footnote{The work was done partially while the author was visiting the Institute for Mathematical Sciences, National University of Singapore in 2016. The visit was supported by the Institute.}}

\author{Lorenzo Carlucci \\
  Dipartimento di Informatica, Sapienza --- Universit\`a di Roma.\\
  \texttt{carlucci@di.uniroma1.it}}

\date{\today{}}

\hypersetup{
  pdfauthor={Lorenzo Carlucci}
  pdftitle ={}
}

\newtheorem{definition}{Definition}
\newtheorem{proposition}{Proposition}

\newtheorem{corollary}{Corollary}

\newtheorem{fact}{Fact}

\newtheorem{open.problem}{Open Problem}

\newtheorem*{theorem*}{Theorem}
\newtheorem*{corollary*}{Corollary}
\newtheorem*{proposition*}{Proposition*}
\newtheorem*{lemma*}{Lemma}
\newtheorem*{fact*}{Fact}
\newtheorem*{claim*}{Claim}
\newtheorem*{open.problem*}{Open Problem}
\newtheorem*{remark*}{Remark}
\newtheorem*{example*}{Example}
\newtheorem*{exercise*}{Exercise}
\newtheorem*{ack*}{Acknowledgment}

\newcommand\Nat{\mathbf{N}}
\newcommand\RCA{\mathbf{RCA}}
\newcommand\ACA{\mathbf{ACA}}
\newcommand\WKL{\mathbf{WKL}}

\newcommand\HT{\mathbf{HT}}
\newcommand\IPT{\mathbf{IPT}}
\newcommand\SIPT{\mathbf{SIPT}}
\newcommand\SRT{\mathbf{SRT}}
\newcommand\RT{\mathbf{RT}}

\newcommand\AHT{\mathbf{AHT}}
\newcommand\HIL{\mathbf{HIL}}

\newcommand\I{\mathbf{I}}
\newcommand\D{\mathbf{D}}

\begin{document} 

\maketitle



\begin{abstract}
Hirst investigated a slight variant of Hindman's Finite Sums Theorem
-- called Hilbert's Theorem -- and proved it equivalent over $\RCA_0$
to the Infinite Pigeonhole Principle for all colors. 
This gave the first example of a natural restriction of Hindman's 
Theorem provably much weaker than Hindman's Theorem itself.
We here introduce another natural variant of Hindman's Theorem 
-- which we name the Adjacent Hindman's Theorem -- and prove 
it to be provable from Ramsey's Theorem for pairs and strictly stronger than 
Hirst's Hilbert's Theorem. The lower bound is obtained by a direct combinatorial 
implication from the Adjacent Hindman's Theorem to the Increasing Polarized Ramsey's 
Theorem for pairs introduced by Dzhafarov and Hirst. 
In the Adjacent Hindman's Theorem homogeneity is required only for 
finite sums of adjacent elements. 
\end{abstract}

\section{Introduction and Motivation}
The strength of Hindman's Theorem is a major open problem in Reverse Mathematics 
(see, e.g., \cite{Mon:11:open}). 
Letting $\HT$ denote the natural formalization of Hindman's Finite Sums Theorem
in the language of arithmetic, the only known upper and lower bounds are the following, 
established thirty years ago by Blass, Hirst and Simpson in \cite{Bla-Hir-Sim:87}:
$$ \ACA_0^+ \geq \HT \geq \ACA_0.$$
Recall that $\ACA_0$ is equivalent to $\RCA_0+\forall X \exists Y (Y=X')$ and that
$\ACA_0^+$ is equivalent to $\RCA_0 + \forall X \exists Y (Y = X^{(\omega)})$. 
As is often the case, the above Reverse Mathematical results are corollaries of 
the following computability-theoretic lower and upper bounds on the complexity of solutions to computable instances of Hindman's Theorem. The following results
are also from \cite{Bla-Hir-Sim:87}:

\begin{enumerate}
\item There exists a computable coloring $c:\Nat\to 2$ such that any solution to 
Hindman's Theorem for $c$ computes $\emptyset'$. 
\item For every computable coloring $c:\Nat\to 2$ there exists a solution set
computable from $\emptyset^{(\omega+1)}$.
\end{enumerate}

Recently there has been some interest in the strength of restrictions of Hindman's Theorem (see \cite{Hir:12:HvsH,DJSW:16}). 

Interestingly, Dzhafarov, Jockusch, Solomon and Westrick \cite{DJSW:16} proved that the only known lower bound on Hindman's Theorem already hits for $\HT^{\leq 3}_4$ (Hindman's Theorem restricted to $4$-colorings and sums of at most $3$ terms) and that $\HT^{\leq 2}_2$ (Hindman's Theorem restricted to $2$-colorings and sums of at most $2$ terms) is unprovable in $\RCA_0$). However, no upper bounds other than those known for the full Hindman's Theorem are known for $\HT^{\leq 2}_2$, let alone $\HT^{\leq 3}_4$. Indeed, it is an open question in Combinatorics whether Hindman's Theorem for sums of at most $2$ terms is already equivalent to the full Hindman's Theorem (see \cite{Hin-Lea-Str:03}, Question 12).

On the other hand, Hirst~\cite{Hir:12:HvsH} investigated a natural restriction of Hindman's Theorem for which a better upper bound can be proved. The variant in question is the following
\begin{definition}[Hilbert's Theorem, \cite{Hir:12:HvsH}]
We denote by $\HIL$ the following principle: Suppose $f:\Nat^{<\Nat}\to k$
is a finite coloring of the finite subsets of the natural numbers. 
Then there is an infinite sequence $\langle X_i\rangle_{i\in\Nat}$
of distinct finite sets and a color $c<k$ such that for every finite
set $F\subset\Nat$ we have $f(\bigcup_{i\in F}X_i)=c$.
\end{definition}
Hilbert's Theorem arises from the Finite Unions Theorem by dropping the condition
that the sequence of sets is such that that $\max(X_i)<\min(X_{i+1})$ for all 
$i\in\Nat$.\footnote{Compare this with our Apartness Condition below.}
Hirst proved that $\HIL$ is equivalent to $\RT^1$ over $\RCA_0$ and therefore is much weaker than Hindman's Theorem. In particular $\HIL$ is finitistically reducible~\cite{Sim:SOSOA}, while Hindman's Theorem is not. 

We introduce another natural variant of Hindman's Theorem, called the Adjacent Hindman Theorem  
and show the following points:
\begin{enumerate} 
\item The Adjacent Hindman's Theorem for $2$ colors -- $\AHT_2$ -- is provable from 
$\RT^2_2$ and hence much weaker than Hindman's Theorem, and
\item $\AHT_2$ implies the Stable Ramsey's Theorem for pairs and $2$ colors, 
and is thus strictly stronger than $\HIL$. 
\end{enumerate}
For the version of the Adjacent Hindman's Theorem for arbitrary finite colorings -- $\AHT$ -- we obtain, 
by the same proofs, that $\AHT$ is provable from $\forall r \RT^2_r$, and
that $\AHT$ implies the Stable Ramsey's Theorem for pairs and arbitrary finite colorings. 
In particular $\AHT$ implies $B\Sigma_3$. 

\section{The Adjacent Hindman's Theorem}
If $n=2^{t_1}+\dots+2^{t_p}$ with $t_1 < \dots < t_p$ let $\lambda(n)=t_1$ and 
$\mu(n)=t_p$, as in~\cite{Bla-Hir-Sim:87}. We consider the following natural variant of Hindman's Theorem: the solution set $H$ is required to be monochromatic only for sums of adjacent elements (with respect
to the increasing enumeration of $H$) and to satisfy the following Apartness
Condition:

\begin{definition}[Apartness Condition] 
A set $X=\{x_1,x_2,\dots\}_<$ satisfies the Apartness Condition (or {\em is apart}) if for all $x,x'\in X$ such that $x < x'$, we have $\mu(x)<\lambda(x')$.
\end{definition}
We use $AS(H)$ (the set of adjacent sums of elements of $H$)
to denote the set of all finite sums of distinct adjacent 
elements of $H=\{h_1,h_2,\dots,\}_<$, where two elements $h$ and $h'$ with 
$h<h'$ are adjacent in $H$
if there is no other element of $H$ between $h$ and $h'$.

\begin{definition}[Adjacent Hindman's Theorem]
$\AHT_k$ is the following principle: For every $c:\Nat\to k$ there exists an infinite set
$H=\{h_1,h_2,h_3\dots,\}_<$ such that all elements of $AS(H)$
have the same $c$-color. Furthermore, $H$ satisfies the  
Apartness Condition\footnote{The Apartness
Condition can be in some cases dropped at the cost of using more colors. For the present discussion
we preferred to include it in the statement of the Adjacent Hindman's Theorem since it typically
simplifies the proofs, and is for free if $\RT^2_2$ is assumed.}.
$\AHT$ denotes $\forall k \AHT_k$.
\end{definition}

Obviously we can define $AS^{\leq n}(H)$ and $AS^{=n}(H)$ with the intuitive meaning, and 
corresponding Hindman-type theorems. 

The Adjacent Ramsey Principles couple two features: they guarantee homogeneity for sums of arbitrary 
length, but severely constrain the way the terms of these sums are chosen.

\section{Upper Bound: Adjacent Hindman's Theorem follows from Ramsey for pairs}
We first show that it is very easy to establish an upper bound on $\AHT_2$ and $\AHT$. 
This should be contrasted with the case of Hindman's Theorem restricted to 
sums of at most two terms ($\HT^{\leq 2}_2$ in the notation of \cite{DJSW:16}), 
for which no upper bound other than $\ACA_0^+$ is currently known. 

\begin{proposition}\label{prop:rt22_to_aht_2} Over $\RCA_0$,
$\RT^2_2$ implies $\AHT_2$.
\end{proposition}

\begin{proof}
Fix a coloring $c:\Nat \to 2$. This induces a coloring $f$ of $[\Nat]^2$
in $2$ colors by setting $f(i,j):=c(2^{i+1}+\dots+2^{j-1}+2^j)$. 
By $\RT^2_2$ let $J=\{j_1,j_2,\dots\}_<$ be 
an infinite homogeneous set for $f$, of color $i<2$. Consider the set 
$$H=\{ (2^{j_1+1}+\dots+2^{j_2}),(2^{j_2+1}+\dots+2^{j_3}),\dots,(2^{j_n+1}+\dots+2^{j_{n+1}}),\dots\}.$$
We claim that $H$ satisfies $\AHT$ for $c$. First, $c(2^{j_n+1}+\dots+2^{j_{n+1}})=f(j_n,j_{n+1})=i$.
Secondly, consider an arbitrary finite sum of adjacent elements of $H$:
$$s= (2^{j_n+1}+\dots+2^{j_{n+1}}) + (2^{j_{n+1}+1}+\dots+2^{j_{n+2}}) + \dots + (2^{j_{n+t}+1}+\dots+2^{j_{n+t+1}}).$$
We have that $c(s)=f(j_n,j_{n+t+1})=i$. Finally, it is obvious that $H$ satisfies the Apartness Condition. 
\end{proof}

Obviously the above proof can be used to show over $\RCA_0$ that $\RT^2_k$ implies $\AHT_k$
uniformly in $k$ so we have the following corollary.

\begin{corollary}
Over $\RCA_0$, $\RT^2$ implies $\AHT$.
\end{corollary}


\section{Lower Bound: Adjacent Hindman's Theorem implies Increasing Polarized Ramsey's Theorem for pairs}

In this section we prove a direct implication from the Adjacent Hindman's Theorem 
for $k$-colorings to the Increasing Polarized Ramsey's Theorem for pairs and $k$-colorings, 
for any $k$. This yields some lower bounds on $\AHT_k$ and on $\AHT$. 
Note that $\AHT_2$ is finitistically reducible (in the sense of Simpson's) since it follows from $\RT^2_2$ 
(see \cite{Pat-Yok:16} for a proof that Ramsey for pairs is finitistically reducible).

The following version of Ramsey's Theorem is introduced in \cite{Dza-Hir:11}.

\begin{definition}[Increasing Polarized Ramsey Theorem]
$\IPT^n_k$ is the following principle: for every $f:[\Nat]^n\to k$ there exists a sequence
$\langle H_1,\dots,H_n\rangle$ of infinite sets such that there exists $c < k$ such that
for all increasing tuple $(x_1,\dots,x_n)\in H_1\times\dots\times H_n$ we have $f(x_1,\dots,x_n)=c$.
The sequence $\langle H_1,\dots,H_n\rangle$ is called increasing p-homogeneous for $f$.
\end{definition}

We first show that $\AHT_2$ implies $\IPT^2_2$. We mention without proof 
that the same implication can be proved for $\HT^{\leq 2}_4$.

\begin{proposition}\label{prop:AHT2_to_IPT22}
Over $\RCA_0$, $\AHT_2$ implies $\IPT^2_2$.
\end{proposition}

\begin{proof}

Let $f:[\Nat]^2\to 2$ be given. Define $g:\Nat\to 2$ as follows.

$$
g(n):=
\begin{cases}
f(\lambda(n),\mu(n)) & \mbox{ if  } \lambda(n) \neq \mu(n),\\
0 & \mbox{ otherwise.}\\
\end{cases}
$$

Let $H$ witness $\AHT_2$ for $g$: $H$ is an infinite set satisfying the Apartness Condition and 
such that $AS(H)$ is monochromatic under $g$. Let the color be $c\in\{0,1\}$. 

Let 
$$H_1 := \{\lambda(n)\,:\, n \in H\}$$ 
and 
$$H_2 := \{\mu(n)\,:\, n \in H\}.$$ 

We claim that $\langle H_1,H_2\rangle$ is increasing p-homogeneous for $f$. 

First observe that, letting $H=\{h_1,h_2,\dots\}_<$, we have
$H_1 = \{ \lambda(h_1),\lambda(h_2),\dots\}_<$ and 
$H_2 = \{ \mu(h_1),\mu(h_2),\dots\}_<$. 
This is so because $\lambda(h_1)\leq\mu(h_1)<\lambda(h_2)\leq\mu(h_2)<\dots$ by the Apartness Condition.

Then we claim that $f(x_1,x_2) = c$ for every increasing pair $(x_1,x_2)\in H_1\times H_2$.
Note that $f(x_1,x_2) = f(\lambda(h_i),\mu(h_j))$ for some $i \leq j$. Note that if $i=j$ then $\lambda(h_i)<\mu(h_i)$
else the pair is not strictly increasing. Since $AS(W)$ is homogeneous, we
have
$$c=g(h_i)=g(h_i+h_{i+1})=g(h_i+h_{i+1}+h_{i+2})=\dots= g(h_i+h_{i+1}+\dots + h_{j-1}+h_j).$$
Now, if $i=j$, then 
$$ f(x_1,x_2) = f(\lambda(h_i),\mu(h_i))=g(h_i)=c.$$
If $i < j$, then 
$$ f(x_1,x_2) = f(\lambda(h_i),\mu(h_j))=g(h_i+h_{i+1}+\dots+h_{j-1}+h_j)=c.$$
since $\lambda(h_i+h_{i+1}+\dots + h_{j-1}+h_j)=\lambda(h_i)$ and $\mu(h_i+h_{i+1}+\dots + h_{j-1}+h_j)=\mu(h_j)$.
Hence in any case $f(x_1,x_2)=c$, as needed.
This shows that $\langle H_1,H_2\rangle$ is increasing p-homogeneous of color $c$ for $f$.
\end{proof}

\begin{corollary}
$\AHT_2$ is strictly stronger than $\HIL$.
\end{corollary}

\begin{proof}
Let $\D^2_2$ be the assertion that for every $\{0,1\}$-valued function $f(x,s)$ such that for 
all $x$ the limit of $f(x,s)$ exists for $s\to \infty$ there is an infinite set $G$ and a color
$j<2$ such that $\lim_s f(x,s)=j$ for all $x\in G$.
By Proposition 3.5 of \cite{Dza-Hir:11} we have that $\RCA_0 \vdash \IPT^2_2\to \D^2_2$. 
From~\cite{Cho-Lem-Yan:10} we have that $\RCA_0\vdash \D^2_2 \to \SRT^2_2$ (where $\SRT^2_2$
is Ramsey's Theorem for pairs for stable colorings), and 
from~\cite{Cho-Joc-Sla:01} that $\SRT^2_2$ is strictly stronger than $B\Sigma_2$. The latter is
equivalent to $\RT^1$ and hence to $\HIL$.
\end{proof}

The above Proposition should be compared with Corollary 2.4 of \cite{DJSW:16}: 
$\RCA_0 + B\Pi^0_1 + \HT^{\leq 2}_2 \vdash \SRT^2_2$. It seems to be unknown 
whether $\SRT^2_2$, or even $\SRT^2_2+B\Sigma_2$ implies $\IPT^2_2$.

It is easy to observe that the proof of Proposition \ref{prop:AHT2_to_IPT22} yields the following
proposition. We denote by $\IPT^2$ the principle $\forall k \IPT^2_k$.

\begin{proposition}\label{prop:AHT_to_IPT2}
$\RCA_0\vdash \forall k(\AHT_k\to\IPT^2_k)$, and 
$\RCA_0\vdash \AHT \to \IPT^2$. 
\end{proposition}

We have the following corollary, improving on the results of Section \ref{sec:Sigma2}.

\begin{corollary}
$\RCA_0 + \AHT \vdash B\Sigma^0_3$.
\end{corollary}

\begin{proof}
By Corollary 11.5 of \cite{Cho-Joc-Sla:01} we have that 
$\RCA_0 + \SRT^2 \vdash B\Sigma^0_3$. By Proposition 3.3 of \cite{Dza-Hir:11} 
we have $\RCA_0 + \IPT^2 \vdash \SRT^2$. 
By Proposition \ref{prop:AHT_to_IPT2} we have $\RCA_0 + \AHT \vdash \IPT^2$. 
This concludes the proof. 
\end{proof}

\section{Adjacent Hindman's Theorem and $\Sigma^0_2$-induction}\label{sec:Sigma2}

In this section we give a direct proof that the Adjacent Hindman's Theorem implies $\Sigma^0_2$-induction.
The proof -- perhaps interestingly -- is an easy adaptation of a recent proof 
by Kolodziejczyk et alii \cite{Kol-Mic-Pra-Skr:16} 
showing that the Ordered Ramsey Theorem implies $\I\Sigma^0_2$.

\begin{proposition} Over $\RCA_0$,
$\AHT$ implies $\I\Sigma^0_2$.
\end{proposition}

\begin{proof}
Let $\phi(x)$ be $\Pi^0_2$: 
$$ \phi(x) \equiv \forall y \exists z A(x,y,z).$$
Suppose $\phi(0)$ and $\forall x (\phi(x)\to \phi(x+1))$ hold. 
We prove that $\phi(a)$ holds. 

Let $D:\Nat\to [0,a+1]$ be defined as follows.
$$ D(n) := \max\{ x\leq a+1 \,:\, \forall x' < x \forall y < \lambda(n)\exists z \leq \mu(n) A(x',y,z)\}.$$

Let $H$ be a witness of $\AHT_{a+2}$ for $D$. 
Let the color of $AS(H)$ under $D$ be $m$.

\begin{fact} For $x \leq a+1$, if for all $x' < x$ we have $\phi(x')$, then $m \geq x$.
\end{fact}

\begin{proof}
For all $i\in H$ and $x'<x$ we have that $\phi(x')$ implies $\forall y < \lambda(i) \exists z A(x',y,z)$. 
By two applications of $\Sigma^0_1$-collection there exists a global bound $v$ such that 
$$ \forall x' < x \forall y < \lambda(i) \exists z \leq v A(x',y,z).$$
Since $H$ is infinite there exists $j\in H$ such that $\mu(j)\geq v$ and $j > i$.
Then we have
$$ \forall x' < x \forall y < \lambda(i) \exists z \leq \mu(j) A(x',y,z).$$
Since $D(i+(i+1)+\dots+(j-1)+j)=m$ we have that
$m$ is the maximum in $[0,a+1]$ such that:
$$ \forall x'< m \forall y < \lambda(i+(i+1)+\dots+(j-1)+j) \exists z \leq \mu(i+(i+1)+\dots+(j-1)+j)A(x',y,z).$$
Since $H$ satisfies the Apartness Condition we have that $m$ is the maximum in $[0,a+1]$ such that:
$$ \forall x'< m \forall y < \lambda(i) \exists z \leq \mu(j)A(x',y,z).$$
Therefore $m \geq x$.
\end{proof}

\begin{fact} For any $x'< m$, $\phi(x')$ holds. 
\end{fact}

\begin{proof}
Take $x' < m$, and any $y$. Since $H$ is infinite there exists $i > y$
such that $i\in H$. Since $D(i)=m$ we have that 
$$\forall y < \lambda(i) \exists z \leq \mu(i) A(x',y,z).$$
Thus there exists $z$ such that $A(x',y,z)$. 
\end{proof}

Now reason as follows. If $m=a+1$ then by Fact 2 $\phi(a)$ holds. 
Suppose $m-1 < a$. Since $\phi(0)$ holds, by Fact 1 we have
$m \geq 1$. By Fact 2 for all $x'<m$, $\phi(x')$ holds. 
By inductive assumption, since $\phi(m-1)$ holds, we know that
$\phi(m)$ holds. So for all $x < m+1$, $\phi(x')$ holds. 
By Fact 1 then $m \geq m+1$, which is impossible.
\end{proof}

We can also give the following short proof. As shown by Kolodziejczyk et alii
in \cite{Kol-Mic-Pra-Skr:16}
failure of $\Sigma^0_2$-induction implies the existence of an $a\in\Nat$ and of 
an infinite word $\alpha \in \{0,\dots,a+1\}^\Nat$ such that there exists no 
highest letter $i$ that appears infinitely often in $\alpha$. Let $D:\Nat \to [0,a+1]$
be defined as follows. 
$$D(n) = \max\{\alpha(k)\,:\, \lambda(n)\leq k \leq \mu(n)\}.$$
Let $H$ be an infinite set witnessing $\AHT_{a+2}$ for $D$. 
Let the color of $H$ under $D$ be $m$. Then for all $i < j$ in $H$ we have 
$$ D(i+\dots+j) = \max \{\alpha(k)\,:\, \lambda(i+\dots + j)\leq k \leq \mu(i+\dots + j)\} = \{\alpha(k)\,:\, \lambda(i)\leq k \leq \mu(j)\} = m.$$
Therefore $m$ is the highest letter occurring infinitely often in $\alpha$.

\section{Conclusions}
We conclude with a speculation: Blass \cite{Bla:05} conjectured that the strength of 
Hindman's Theorem might be growing with the length of the finite sums whose homogeneity
is guaranteed. The case of the Adjacent Hindman's Theorem might indicate that a measure of complexity for
Hindman's Theorem should not only consider the length of the sums but -- more importantly --
the {\em structure} according to which the elements of the sums are picked in the 
homogeneous set. This idea can be appropriately formalized in terms of trees labeled 
by integers and gives rise to a family of Hindman-type principles that might deserve
attention.


\begin{thebibliography}{1}

\bibitem{Bla:05}
A.~Blass.
Some questions arising from Hindman's Theorem.
\textit{Sci. Math. Jpn.}, 62, 331--334, 2005.

\bibitem{Bla-Hir-Sim:87}
A.R.~Blass, J.L.~Hirst, S.G.~Simpson: Logical analysis of some theorems of combinatorics
and topological dynamics. In: {\em Logic and combinatorics} (Arcata, Calif., 1985),
Contemp. Math., vol. 65, pp. 125--156. Amer. Math. Soc., Providence, RI (1987).


\bibitem{Cho-Joc-Sla:01}
P.~Cholak, C.G.~Jockusch, T.~Slaman. 
On the strength of Ramsey's Theorem for pairs. 
J. Symbolic Logic 66 (2001), no. 1, 1--55.

\bibitem{Cho-Lem-Yan:10}
\newblock C.~T.~Chong, S.~Lempp, and Y.~Yang, 
\newblock On the role of the collection principle for
$\Sigma^0_2$ formulas in second-order reverse mathematics, 
\newblock {\em Proceedings of the American Mathematical Society}, 138 (2010), 1093--1100.

\bibitem{Dow-Hir-Lem-Sol:01}
R.~Downey, D.R.~Hirschfeldt, S.~Lempp, R.~Solomon. 
A $\Delta^0_2$ set with no infinite low subset
in either it or its complement. 
J. Symbolic Logic 66 (2001), no. 3, 1371--1381.

\bibitem{Dza-Hir:11}
D.D.~Dzhafarov and J.L.~Hirst.
\newblock\ The polarized Ramsey's theorem.
\newblock\ {\em Archive for Mathematical Logic}, 48(2):141--157, 2011.

\bibitem{DJSW:16}
D.~Dzhafarov, C.~Jockusch, R.~Solomon, L.B.~Westrick. 
Effectiveness of Hindman's Theorem for bounded sums. 
In A.~Day, M.~Fellows, N.~Greenberg, B.~Khoussainov, and A.~Melnikov, eds., 
{\em Proceedings of the International Symposium on Computability and Complexity} 
(in honour of Rod Downey's 60th birthday), 
Lecture Notes in Computer Science, Springer, to appear.


\bibitem{Hin-Lea-Str:03}
N.~Hindman, I.~Leader, and D.~Strauss. 
Open problems in partition regularity.
Combinatorics Probability and Computing 12 (2003) 571--583.

\bibitem{Hir:12:HvsH}
J.~Hirst.
\newblock\ Hilbert vs. Hindman. 
\newblock\ {\em Archive for Mathematical Logic 51}, (1-2):123--125, 2012. 

\bibitem{Kol-Mic-Pra-Skr:16} L.~A.~Kolodziejczyk, H.~Michalewski, P.~Pradic, M.~Skrzypczak, . 
The logical strength of B\"uchi's decidability theorem, preprint. 
Preprint, arXiv:1608.07514 [cs.LO], 2016.

\bibitem{Lep:pc}
Francesco Carlo Lepore, personal communication.

\bibitem{Mon:11:open}
\newblock\ A.~Montalb\'{a}n.
\newblock\ {\em Bulletin of Symbolic Logic}, Volume 17, Issue 3 (2011), 431--454.

\bibitem{Pat-Yok:16}
L.~Patey, K.~Yokoyama.
{\em The proof theoretic strength of Ramsey's Theorem for pairs}. 
Preprint, 2016.

\bibitem{Sim:SOSOA}
S.G.~Simpson: {\em Subsystems of second order arithmetic}, second edn. Perspectives
in Logic. Cambridge University Press, Cambridge (2009).

\end{thebibliography}
\end{document}